\documentclass[11 pt]{article}
\usepackage{amsmath, amsthm, amssymb,verbatim,fullpage,multicol,booktabs}
\usepackage{graphicx}
\usepackage{hyperref}
\usepackage[table]{xcolor}
\usepackage[T1]{fontenc}
\usepackage{marvosym}
\usepackage{appendix}
\usepackage[utf8]{inputenc}
\usepackage{braket}
\usepackage{authblk}
\usepackage{multirow}
\usepackage{enumerate}
\usepackage{enumitem}
\usepackage{caption}
\usepackage{subcaption}
\usepackage{ mathrsfs }
\usepackage{float}
\usepackage{tikz}
\usepackage{tkz-graph}
\tikzstyle{vertex}=[circle, draw, inner sep=0pt, minimum size=6pt]

\usepackage{array}
\usepackage{multirow}
\newcolumntype{M}[1]{>{\centering\arraybackslash}m{#1}}
\newcolumntype{N}{@{}m{0pt}@{}}

\newcommand{\rt}{\ell}

\newcommand{\E}{E_C}
\newcommand{\EB}{E_B}

\newcommand{\hroot}{\tilde{\alpha}}

\newtheorem{theorem}{Theorem}[section]
\newtheorem{lemma}[theorem]{Lemma}
\newtheorem{proposition}[theorem]{Proposition}
\newtheorem{corollary}{Corollary}

\theoremstyle{definition}
\newtheorem{definition}[theorem]{Definition}
\newtheorem{example}[theorem]{Example}

\newtheorem*{theorem*}{Theorem}
\newtheoremstyle{named}{}{}{\itshape}{}{\bfseries}{.}{.5em}{\thmnote{#3 }#1}
\theoremstyle{named}

\def\A{{ \mathbf u}}
\def\B{{\mathbf v}}
\def\P{{\mathscr{P}}}

\def\CC{{\mathbb C}}

\def\hr{\tilde{\alpha}}

\everymath{\displaystyle}

\allowdisplaybreaks

\title{The $q$-analog of Kostant's partition function and the highest root of the classical Lie algebras}
 \author[1]{Pamela E. Harris\thanks{\textcolor{blue}{\href{mailto:Pamela.Harris@usma.edu}{Pamela.Harris@usma.edu}}. This research was performed while the author held a National Research Council Research Associateship Award at USMA/ARL.}}
 \author[2]{Erik Insko\thanks{\textcolor{blue}{\href{mailto:einsko@fgcu.edu}{einsko@fgcu.edu}}}}
 \author[3]{Mohamed Omar\thanks{\textcolor{blue}{\href{mailto:omar@g.hmc.edu}{omar@g.hmc.edu}} This research was supported by research funds from Harvey Mudd College.}}
 \affil[1]{\href{http://www.usma.edu/math/SitePages/Home.aspx}{Department of Mathematical Sciences, United States Military Academy}}
 \affil[2]{\href{http://www.fgcu.edu/CAS/Departments/Math/12490.asp}{Department of Mathematics, Florida Gulf Coast University}}
 \affil[3]{\href{https://www.math.hmc.edu}{Department of Mathematics, Harvey Mudd College}}

\begin{document}
  \maketitle
\begin{abstract}
Kostant's partition function counts the number of ways to represent a particular vector (weight) as a nonnegative integral sum of positive roots of a Lie algebra. For a given weight
the $q$-analog of Kostant's partition function is a polynomial where the coefficient of $q^k$  is the number of 
ways 
the weight can be written as a nonnegative integral sum of exactly $k$ positive roots. In this paper we determine generating 
functions for the $q$-analog of Kostant's partition function when the weight in question is the highest root of the classical Lie 
algebras of types $B$, $C$ and $D$.

\end{abstract}


\section{Introduction}

In this paper we focus on finding an explicit formula for a partition counting problem in the setting of combinatorial representation theory of finite-dimensional classical Lie algebras. The problem of interest involves Kostant's partition function, whose function values are given by the coefficients of the power series expansion of the product
\[\prod_{\alpha\in\Phi^+}\left(\frac{1}{1-e^{-\alpha}}\right)\]
where $\Phi^+$ denotes a set of positive roots for a classical Lie algebra $\mathfrak{g}$. In combinatorial terms, Kostant's 
partition function counts the number of ways one can express a vector, referred to as a weight, as an integral 
sum of positive roots. Given a weight $\xi$, we denote this count by $\wp(\xi)$.

As Kostant's partition function is an integral part of Kostant's weight multiplicity formula, 
previous bodies of work have involved combinatorial \emph{approaches to computing} the value of the partition function and weight 
multiplicities 
\cite{BZ2,BGR,deckhart,FP,Gupta,PH,PH2,PH3,Tate}. In addition, there have been notable bodies of work which relate Kostant's partition 
function to 
flow polytopes and volume computation of polytopes
\cite{BBCV,Baldoni,MM}.
However, as stated by Baldoni and Vergne, ``[i]n general, it is difficult to give ``concrete'' formulae for the partition functions'' 
\cite{Baldoni} and much work has be done to construct efficient programs to compute their values \cite{BBCV,cochet,BS}. In fact, 
Kostant's partition function is only one from a class of more general 
\emph{vector partition functions}. As Beck and Robins point out, vector partition functions ``have many interesting properties 
and 
give rise to intriguing open questions,''   \cite{BeckRobins}.

The connection to representation theory stems from Kostant's use of his partition function to give another expression of the Weyl 
character formula, which gives the character of an irreducible representation $V$ of a complex semi-simple Lie algebra 
$\mathfrak{g}$. The Weyl character formula is given by
\begin{align*}
\text{ch}(V)&=\frac{\sum_{w\in W}(-1)^{\ell(w)}e^{w(\lambda+\rho)}}{e^\rho\prod_{\alpha\in \Phi^+}(1-e^{-\alpha})} .
\end{align*}

We now recall that Kostant's weight multiplicity formula gives the multiplicity of the weight $\mu$ in the finite-dimensional 
representation with highest weight $\lambda$ is given by
\begin{align}
m(\lambda,\mu)&=\sum_{w\in W}(-1)^{\ell(w)}\wp(w(\lambda+\rho)-(\mu+\rho)),\label{KWMF}
\end{align}
where $\wp$ denotes Kostant's partition function~\cite{KMF}. However, using Kostant's weight multiplicity formula to compute 
weight 
multiplicities is often extremely cumbersome. The first difficulty in using Equation \eqref{KWMF} to compute weight multiplicities 
is that the order of the Weyl group grows factorially as the Lie algebra's rank increases. Secondly, for fixed $\lambda$ and $\mu$ 
it is unknown for which elements $w$ of the Weyl group, the value of $\wp(w(\lambda+\rho)-(\mu+\rho))$ is strictly positive. 
However, in 2015 Harris, Insko, and Williams (for all finite-dimensional classical Lie algebras) described and enumerated the 
elements of the Weyl group which provide positive partition 
function values in the cases where $\lambda$ is the highest root and $\mu$ is the zero weight \cite{HIW}. 
These sets of Weyl group elements were called \emph{Weyl alternation sets}.

The reason one is interested in this computation is to lay the foundation for approaching a much more difficult question: Can one use purely combinatorial techniques to prove that for any finite-dimensional simple Lie algebra $\mathfrak{g}$,
\begin{align}m_q(\tilde{\alpha},0)&=\displaystyle\sum_{w\in W}(-1)^{\ell(w)}\wp_q(w(\tilde{\alpha}+\rho)-\rho)=q^{e_1}+q^{e_2}+\cdots+q^{e_r},\end{align}
where $\tilde{\alpha}$ is the highest root and $e_1,e_2,\ldots,e_r$ denote the exponents of $\mathfrak{g}$?
The above result stems from work of Kostant and the definition of the $q$-analog  provided by Lusztig in~\cite{Kostant}, 
\cite[Section 10, p.~226]{LL}. We note that a purely combinatorial proof for Lie type $A$ 
was given by the first author in 2011~\cite{PH}.

The goal of this paper is to continue on the quest to recover the exponents of the classical Lie algebras via a purely combinatorial proof of Kostant's result. 
However, as we have previously stated, the remaining road block is to find closed formulas for the value of the partition function on the elements of the Weyl 
alternation sets.
With this aim in mind, we focus on determining $\wp_q(w(\tilde{\alpha}+\rho)-\rho)$ in the case where $w$ is the identity element 
of the Weyl group, $\hroot$ is the highest root of the Lie algebra $\mathfrak{g}$, and where $\wp_q$ denotes the $q$-analog of 
Kostant's partition function as introduced by Lusztig as 
the polynomial-valued function
\begin{align}\wp_q(\xi)=a_0+a_1q+a_2q^2+a_3q^3+\cdots+a_kq^k\end{align}
where $a_i$ is the number of ways to write the weight $\xi$ as a sum of exactly $i$ positive roots~\cite{LL}. 

In this work we are specifically interested in the value of the $q$-analog of Kostant's partition function 
when the weight in question is the highest root of a finite-dimensional classical Lie algebra. 
We remark that the type $A_r$ case yields $\wp_{A_r}(\hroot)=q(1+q)^{r-1}$.
We provide generating formulas for the value of $\wp_q(\tilde{\alpha})$ in the case where $\hroot$ is the highest root of the Lie algebras $B_r$, $C_r$,
and $D_r$. To specify which Lie algebra we are considering we let $\P_{B_r}(q)$, $\P_{C_r}(q)$, and $\P_{D_r}(q)$ denote $\wp_q(\hroot)$, 
in the Lie algebras of type $B_r$, $C_r$, and $D_r$, respectively. Hence the main results of this paper (Theorems 
\ref{Bgenfun}, \ref{Cgenfun}, and \ref{thm:DGenFun}) states:

\begin{theorem*}[Generating Functions]\label{mainresult} The closed formulas for the generating functions\\ $\textstyle\sum_{r\geq 1}\P_{B_{r}}(q)x^r$, $\textstyle\sum_{r\geq 1}\P_{C_{r}}(q)x^r$, and $\textstyle\sum_{r\geq 4}\P_{D_{r}}(q)x^r$,  are given by
\begin{align*}
\displaystyle\sum_{r\geq 1}\P_{B_r}(q)x^r&=\frac{qx+(-q-q^2)x^2+q^2x^3}{1-(2+2q+q^2)x+(1+2q+q^2+q^3)x^2},\\
\displaystyle\sum_{r\geq 1}\P_{C_r}(q)x^r&=\frac{qx+(-q-q^2)x^2}{1-(2+2q+q^2)x+(1+2q+q^2+q^3)x^2},\\
\sum_{r \geq 4} \P_{D_r}(q)x^r &= \dfrac{(q+4q^2+6q^3+3q^4+q^5)x^4 - (q+4q^2+6q^3+5q^4+3q^5+q^6)x^5}{1-(2+2q+q^2)x+(1+2q+q^2+q^3)x^2}.
\end{align*}
\end{theorem*}
Using classical techniques from generating functions we extract the following explicit formulas.
\begin{corollary}[Explicit Formulas]\label{corclosed}
The following are explicit formulas for the value of the $q$-analog of Kostant's partition function on the highest root of the classical Lie algebras:
\begin{align*}
\text{Type $B_r$}\; (r\geq 2):&\hspace{3mm}\P_{B_r}(q)=\alpha_1(q) \cdot \left(\beta_1(q)\right)^{r-2} + \alpha_2(q) \cdot \left(\beta_2(q)\right)^{r-2},\\
\text{Type $C_r$}\; (r\geq 1):&\hspace{3mm}\P_{C_r}(q)=\gamma_1(q) \cdot \left(\beta_1(q)\right)^{r-1} + \gamma_2(q) \cdot \left(\beta_2(q)\right)^{r-1},\\
\text{Type $D_r$}\; (r\geq 4):&\hspace{3mm}\P_{D_r}(q)=\delta_1(q) \cdot \left(\beta_1(q)\right)^{r-4} + \delta_2(q) \cdot \left(\beta_2(q)\right)^{r-4},
\end{align*}
where 
\[
\beta_1(q) = \dfrac{2+2q+q^2+q\sqrt{q^2+4}}{2}, \hspace{0.1in} \beta_2(q) = \dfrac{2+2q+q^2-q\sqrt{q^2+4}}{2}
\]
and 
\begin{align*}
\alpha_1(q) &= q \cdot \dfrac{q^4 +q^3\left(\sqrt{q^2+4}+1\right)+ q^2\left(\sqrt{q^2+4}+5\right) + q\left(3\sqrt{q^2+4}+4\right)+2\left(\sqrt{q^2+4}+2\right)}{2(q^2+4)}, \\
\alpha_2(q) &= q \cdot \dfrac{q^4- q^3\left(\sqrt{q^2+4}-1\right)-q^2\left(\sqrt{q^2-4}-5\right)+q\left(4-3\sqrt{q^2+4}\right)-2\left(\sqrt{q^2+4}-2\right) }{2(q^2+4)},\\
\gamma_1(q) &= \dfrac{q\left(q^2 + q\sqrt{q^2+4}+4\right)}{2(q^2+4)}, \hspace{5mm} \gamma_2(q) = \dfrac{q\left(q^2 - q\sqrt{q^2+4}+4\right)}{2(q^2+4)},\\
\delta_1(q) &=q\cdot\frac{ 2 + 9 q + 12 q^2 + 8 q^3 + 3 q^4 + q^5 +  (1 + 4 q + 6 q^2 + 3 q^3 + q^4)\sqrt{q^2+4}}{2 \sqrt{
 q^2+4}},\\
\delta_2(q) &=q\cdot\frac{-2 - 9 q - 12 q^2 - 8 q^3 - 3 q^4 - q^5 +(1 + 4 q + 6 q^2 + 3 q^3 + q^4)\sqrt{q^2+4}}{2\sqrt{q^2+4}}.
\end{align*}
\end{corollary}

The proof of Corollary~\ref{corclosed} is provided in Appendix~\ref{appendix}. We remark that evaluating $q=1$ in the above explicit functions yields closed formulas for the value of Kostant's partition function on the highest root of the classical Lie algebras considered. These are new closed formulas for the number of ways one can express the highest root of a classical Lie algebra as a nonnegative integral sum of positive roots. 

Observe that for the Lie algebras of type $B$ and $C$ the associated generating functions (with $q=1)$ were used by Butler and Graham to 
count the number of multiplex juggling sequences\footnote{OEIS sequence \href{http://oeis.org/A136775}{A136775}.} of length $n$, 
base state $<1,1>$ and hand capacity 2, and the number of periodic multiplex juggling sequences\footnote{OEIS sequence 
\href{http://oeis.org/A081567}{A081567}.} of length $n$ with base state $<2>$, respectively~\cite{juggle}. We list 
these generating functions in Table~\ref{genfunc}. While the generating functions agree, we do not provide a bijective argument 
for this observation.  Hence we would welcome a proof of this connection.

\begin{table}[h!]
\centering

\begin{tabular}{|c|c|N}
\hline
Type & Generating function&\\ 
\hline
$B$ &$\frac{x-2x^2+x^3}{1-5x+5x^2}$ &\\[20pt]
\hline
$C$&$\frac{x-2x^2}{1-5x+5x^2}$&\\[20pt]
\hline
\end{tabular}
\caption{Generating functions for the value of $\wp(\hroot)$ in Lie types $B$ and $C$}
\label{genfunc}
\end{table}
We note that this is not the first time that the mathematics of juggling provides insight into such computations. For example, 
Ehrenborg and Readdy used juggling patterns as an application to compute $q$-analogs. In 
particular, they used juggling patterns 
to compute the Poincar\'{e} series of the affine Weyl group $\tilde{A}_{r-1}$ \cite[Corollary~4.3]{Readdy}.

This paper is organized as follows: Section \ref{Background} contains the necessary background about Lie algebras to make our approach precise.  Sections \ref{type B} and \ref{type C} contain recursive formulas that count the number of partitions of the highest root $\tilde{\alpha}$ in types $B_r$ and $C_r$ (respectively) as sums of positive roots. These recursions are used to derive the closed forms for the generating functions  $\textstyle\sum_{r \geq 1} \P_{B_r}(q)x^r$ in Theorem \ref{Bgenfun} and $\textstyle\sum_{r \geq 1}^{} \P_{C_r}(q)x^r$ in Theorem \ref{Cgenfun}. Finally in Section \ref{type D} we describe how the total number of partitions of the highest root as sums of positive roots in type $D_r$ can be calculated using the partitions of the highest root in type $B_{r-2}$, the content of Theorem \ref{thm:Dthm}. This leads to the derivation of a closed formula for the generating function $\textstyle\sum_{r \geq 4} \P_{D_r}(q)x^r$ in Theorem \ref{thm:DGenFun}.

\section{Background} \label{Background}
The structure of classical root systems has been extensively studied and completely classified up to isomorphism.  
We follow the conventions and choices of vector space bases set forth in Goodman and Wallach's text,~\cite[Section 2.4.3]{GW}.

Let $\Phi$ be the root system for a Lie algebra $\mathfrak g$ of type $A_r$, $B_r$, $C_r$, or $D_r$.  

\begin{definition} \label{def:simpleroots}
 A subset $\Delta = \{ \alpha_1, \alpha_2, \ldots, \alpha_r \} \subset \Phi$ is a set of \emph{simple roots} if every root in $\beta \in \Phi$ can be written uniquely as
 \[ \beta = c_1 \alpha_1 + c_2 \alpha_2 + \cdots + c_r \alpha_r\]
\end{definition}
\noindent with all the $c_i$'s having the same sign.  Choosing a set of simple roots $\Delta$ partitions $\Phi$ into two disjoint subsets $\Phi = \Phi^+ \cup \Phi^{-}$ of positive roots $\Phi^{+}$ and negative roots $\Phi^{-}$, 
where $\Phi^{+}$ is the collection of roots where $c_i \geq 0$ and $\Phi^{-} $ 
is the set of roots with $c_i \leq 0$ for all $\alpha_i \in \Delta$.  
If $\Delta \subset \Phi$ is a subset of simple roots and $\beta = c_1 \alpha_1 + c_2 \alpha_2 + \cdots + c_r \alpha_r$ is a root, then the \emph{height} of $\beta$ is
\[ ht(\beta) = c_1 + c_2 + \cdots + c_r .\]  Naturally the positive roots are those $\beta \in \Phi$ with $ht(\beta) >0$.  
In the classical Lie algebras, there is a unique \emph{highest root} $\tilde{\alpha} $ defined by the property $ht(\tilde{\alpha}) 
\geq ht(\beta)$ for all $\beta \in \Phi$.

Following the conventions set in \cite[Section 2.4.3]{GW}, we now describe a choice of simple roots, list the positive roots, and specify the corresponding highest root in Lie types $B_r,\ C_r,
$ and $D_r$.\\

\noindent
\underline{\textbf{Type $B_r$ ($\mathfrak{so}_{2r+1}(\mathbb{C})$):}}
Let $r\geq 2$ and let $\Delta = \{\alpha_i \ | \ 1 \leq i \leq r \}$ be a set of simple roots. 
We describe the set of positive roots $\Phi^+$ by breaking them up into roots we refer to as \emph{hooked} and \emph{nonhooked}.  
We define the set of \emph{nonhooked positive roots} to be \[ \Phi_{B_r^{NH}}^+=\Delta\cup\left \{\alpha_i+\alpha_{i+1}+ \cdots +\alpha_j \mid  1\leq i  < j \leq r \right \}  , \]
and we define the set of \emph{hooked positive roots} to be 
\[ \Phi_{B_r^H}^+=\left \{\alpha_i+\alpha_{i+1}+ \cdots + \alpha_{j-1}+2 \alpha_j+2 \alpha_{j+1}+\cdots +2\alpha_r \mid   1\leq i  < j \leq r \right \}  . \]
Note the highest root of type $B_r$ is $\tilde{\alpha}=\alpha_1+2\alpha_2+\cdots+2\alpha_r$ and $\Phi^+= \Phi_{B_r^H}^+ \sqcup 
\Phi_{B_r^{NH}}^+$.  
\vspace{12pt}

\noindent
\underline{\textbf{Type $C_r$ ($\mathfrak{sp}_{2r}(\mathbb{C})$):}}
Let $r\geq 2$ and let $\Delta = \{\alpha_i \ | \ 1 \leq i \leq r \}$, be a set of simple roots. 
Again we break $\Phi^+$ into the set of hooked and nonhooked roots. 
We define the set of \emph{nonhooked positive roots} to be \[ \Phi_{C_r^{NH}}^+=\Delta\cup\left \{\alpha_i+\alpha_{i+1}+ \cdots +\alpha_j \mid   1\leq i  < j \leq r \right \}  , \]
and we define the set of \emph{hooked positive roots} to be 
\[ \Phi_{C_r^H}^+=\left \{\alpha_i+\alpha_{i+1}+ \cdots+ \alpha_{j-1}+2 \alpha_j+2 \alpha_{j+1}+\cdots +2 \alpha_{r-1}+\alpha_r \mid   1\leq i  < j \leq r-1 \right \}  . \]

Note that the highest root is $\tilde{\alpha}=2\alpha_1+2\alpha_2+\cdots+2\alpha_{r-1}+\alpha_r$ and $\Phi^+= \Phi_{C_r^H}^+ \sqcup \Phi_{C_r^{NH}}^+$.
\vspace{12pt}

\noindent
\underline{\textbf{Type $D_r$ ($\mathfrak{so}_{2r}(\mathbb C)$):}}  Let $r\geq 4$ and let $\Delta = \{\alpha_i \ | \ 1 \leq i \leq r \}$ be a choice of a set of simple roots. 
We define the set of \emph{nonhooked positive roots} to be \[ \Phi_{D_r^{NH}}^+=\Delta\cup\left \{\alpha_i+ \cdots +\alpha_j \mid   1\leq i  < j \leq r \right \} \cup \left \{\alpha_i+\alpha_{i+1}+ \cdots +\alpha_{r-2}+\alpha_r \mid   1\leq i  \leq r-2 \right \}  , \]
and we define the set of \emph{hooked positive roots} to be 
\[ \Phi_{D_r^H}^+=\left \{\alpha_i+\alpha_{i+1}+ \cdots \alpha_{j-1}+2 \alpha_j+2 \alpha_{j+1}+\cdots +2 \alpha_{r-2}+\alpha_{r+1}+\alpha_r \mid   1\leq i  < j \leq r-2 \right \}  . \]

Note that the highest root of type $D_r$ is $\tilde{\alpha}= \alpha_1+2\alpha_2+\cdots+2\alpha_{r-2}+\alpha_{r-1}+\alpha_r$ and 
the set of positive roots 
is the disjoint union  $\Phi^+= \Phi_{D_r^H}^+ \sqcup \Phi_{D_r^{NH}}^+$.

\begin{definition} \label{def:partition}
A \emph{partition} of the highest root $\tilde{\alpha}$ in $\Phi^+$ is a multiset $\{ \beta_1, \beta_2, \ldots, \beta_k\} = \Gamma $ such that $\beta_i\in \Phi^{+}$ for all $1\leq i\leq k$ and $\tilde{\alpha} = \beta_1+\beta_2 + \cdots + \beta_k$.  
The elements $\beta\in\Gamma$ are called the \emph{parts} of the partition $\Gamma$.
\end{definition}

\subsection{Coordinate Vector Notation} \label{subsection:coordinate}

Let $e_i = ( 0,0, \ldots, 1 , 0 , 0, \ldots) $ denote the $i$th standard basis vector of $\CC^\infty$.  
For each classical Lie algebra $\mathfrak g$ with set of roots $\Phi_{\mathfrak{g}}$ we let $\Psi: \Phi_{\mathfrak{g}}\rightarrow \CC^\infty$ be the map defined on the simple roots of the Lie algebras of type $B_r$ and $D_r$ via $\Psi(\alpha_i) = e_i$ and for the Lie algebra of type $C_r$ via $\Psi(\alpha_i)=e_{r-i+1}$.  We extend this map linearly to the roots in $\Phi_{\mathfrak{g}}$. In doing this, we relate the coefficients of $\P_{B_r}(q),\ \P_{C_r}(q),\ \P_{D_r}(q)$ respectively to counting the number of ways a specific vector in $\CC^\infty$ can be expressed as a linear combination of vectors from a fixed set, where weights are required to be nonnegative integers, as encapsulated in the following lemmas.  The proofs of these lemmas are immediate.

\begin{lemma}\label{lem:typeB}
Let $r,k \geq 1$ be integers.  The number of partitions of the highest root $\hr$ of the Lie algebra of type $B_r$ with $k$ parts is equal to the number of ways of writing the vector \[ \Psi(\hr) = e_1+2e_2+2e_3+\cdots + 2e_r =( \underbrace{\ 1,\ 2,\ 2,\ \ldots,\ 2}_{r \mbox{ nonzero entries }},0,0,\ldots) \in \mathbb{C}^{\infty} \] as a nonnegative integer combination of $k$ of the following vectors:
\begin{itemize}
\item Nonhooked vectors (parts) of the form $e_i$ with $1 \leq i \leq r$.  
\item Nonhooked vectors (parts) of the form $e_i +e_{i+1}+\cdots + e_j$ with $1 \leq i < j \leq r$.  
\item Hooked vectors (parts) of the the form  $e_i + e_{i+1}+\cdots + e_{j-1} +2e_{j} + 2e_{j+1} +\cdots + 2e_r$ with $1 \leq i < j \leq r$.
\end{itemize}
\end{lemma}

\begin{lemma}\label{lem:TypeC}
Let $r,k \geq 1$ be integers.  The number of partitions of the highest root $\hr$ of the Lie algebra of type $C_r$ with $k$ parts is equal to the number of ways of writing the vector 
\[\Psi(\hr)= e_1+2e_2+\cdots + 2e_{r-1}+2e_{r} = (\underbrace{\ 1, \ 2,\ 2,\ \ldots,\ 2, }_{r \mbox{ nonzero entries }}0,0,\ldots) 
\in \mathbb{C}^{\infty} \]
 as a nonnegative integer combination of $k$ of the following vectors:
\begin{itemize}
\item Nonhooked vectors (parts) of the form $e_i $ with $1 \leq i  \leq r$. 
\item Nonhooked vectors (parts) of the form $e_i +e_{i+1}+\cdots + e_j$ with $1 \leq i \leq j \leq r$. 
\item Hooked vectors (parts) of the form $e_1 + 2e_{2} + \cdots + 2e_{i} +e_{i+1} + \cdots + e_j$ with $1 < i < j \leq r$.  
\end{itemize}
\end{lemma}

\begin{lemma}\label{lem:TypeD}
Let $r,k \geq 1$ be integers.  The number of partitions of the highest root $\hr$ of the Lie algebra of type $D_r$ with $k$ parts is equal to the number of ways of writing the vector 
\[ \Psi(\hr)=e_1+2e_2+2e_3+\cdots + 2e_{r-2}+e_{r-1}+e_r = (\underbrace{\ 1,\ 2,\ 2,\ldots,\ 2,\ 1,\ 1}_{r \mbox{ nonzero entries }},0,0,\ldots) \in \mathbb{C}^{\infty} \] as a nonnegative integer combination of $k$ of the following vectors:
\begin{itemize}
\item Nonhooked vectors (parts) of the form $e_i$ with $1 \leq i \leq r$.  
\item Nonhooked vectors (parts) of the form $e_i +e_{i+1}+\cdots + e_{j}$ with $1 \leq i < j \leq r$. 
\item Nonhooked vectors (parts) of the form $e_i +e_{i+1}+\cdots + e_{r-2}+e_r$ with $1 \leq i  \leq r-2$. 
\item Hooked vectors (parts) of the the form  \[e_i+e_{i+1}+ \cdots+ e_{j-1}+2 e_j+2 e_{j+1}+\cdots +2 e_{r-2}+e_{r+1}+e_r \] with $  1\leq i  < j \leq r-2$.
\end{itemize}
\end{lemma}

We remark that henceforth we use the words \emph{vector} and \emph{part} interchangeably. 

\section{Type $B$} \label{type B}

Let $\P_{B_r}(q):=\wp_q(\hr)$, where $\hr=\alpha_1+2\alpha_2+\cdots+2\alpha_r$ is the highest root of the Lie algebra of type $B_r$. By Lemma~\ref{lem:typeB}, we can recover $\P_{B_r}(q)$ by determining the number of partitions of $e_1+2e_2+\cdots+2e_r$ into $k$ parts where the parts are as specified in Lemma~\ref{lem:typeB}.  In this light, throughout this section, a \emph{partition} of the vector $e_1+2e_2+\cdots+2e_{\ell}$ in $\CC^{\infty}$ where $2 \leq \ell \leq r$ is a nonnegative integer combination of $k$ (not necessarily distinct) vectors from the following set:
 \begin{itemize}
\item Nonhooked vectors (parts) of the form $e_i$ with $1 \leq i \leq \ell$.  
\item Nonhooked vectors (parts) of the form $e_i +e_{i+1}+\cdots + e_j$ with $1 \leq i < j \leq \ell$.  
\item Hooked vectors (parts) of the the form  $e_i + e_{i+1}+\cdots + e_{j-1} +2e_{j} + 2e_{j+1} +\cdots + 2e_{\ell}$ with $1 \leq i < j \leq \ell$.
\end{itemize}
\noindent
The following example illustrates the computation of a partition function.

\begin{example} Let $\hr=e_1+2e_2+2e_3+2e_4$ be the highest root of the Lie algebra of type $B_5$. In Table \ref{tab:example} we 
provide the 40 ways in which $\hr$ can be expressed as a sum of nonhooked and hooked vectors of type $B_4$. 
% Note that in addition we specify the number of vectors used in each of the partitions of $\hr$. 
From this we can compute 
\[\P_{B_4}(q)=q^7+3q^6+8q^5+11q^4+11q^3+5q^2+q. \qed \] 
\end{example}

\begin{table}[h]
\begin{minipage}[t]{.5 \linewidth}
\begin{center}
\begin{tabular}[t]{|l|l|}\hline
$n$& Partitions of $\hr$ using $n$ vectors\\\hline
\hline
7&$\{e_1,\;e_2,\;e_2,\;e_3,\;e_3,\;e_4,\;e_4\}$\\\hline
\multirow{3}{*}{6}&$\{e_1+e_2,\;e_2,\;e_3,\;e_3,\;e_4,\;e_4\}$\\%\cline{2-2}
&$\{e_1,\;e_2,\;e_2+e_3,\;e_3,\;e_4,\;e_4\}$\\%\cline{2-2}
&$\{e_1,\;e_2,\;e_2,\;e_3,\;e_3+e_4,\;e_4\}$\\\hline
\multirow{8}{*}{5}&$\{e_1,\;e_2,\;e_2,\;e_3,\;e_3+2e_4\}$\\%\cline{2-2}
&$\{e_1,\;e_2,\;e_3,\;e_4,\;e_2+e_3+e_4\}$\\%\cline{2-2}
&$\{e_2,\;e_3,\;e_4,\;e_4,\;e_1+e_2+e_3\}$\\%\cline{2-2}
&$\{e_1,\;e_2,\;e_2+e_3,\;e_4,\;e_3+e_4\}$\\%\cline{2-2}
&$\{e_1,\;e_2,\;e_2,\;e_3+e_4,\;e_3+e_4\}$\\%\cline{2-2}
&$\{e_1+e_2,\;e_2,\;e_3,\;e_4,\;e_3+e_4\}$\\%\cline{2-2}
&$\{e_1+e_2,\;e_3,\;e_4,\;e_4,\;e_2+e_3\}$\\%\cline{2-2}
&$\{e_1,\;e_2+e_3,\;e_4,\;e_4,\;e_2+e_3\}$\\\hline
\multirow{7}{*}{4}&$\{e_1,\;e_2,\;e_3,\;e_2+e_3+2e_4\}$\\%\cline{2-2}
&$\{e_1,\;e_2,\;e_2+e_3,\;e_3+2e_4\}$\\%\cline{2-2}
&$\{e_1+e_2,\;e_2,\;e_3,\;e_3+2e_4\}$\\%\cline{2-2}
&$\{e_2,\;e_3,\;e_4,\;e_1+e_2+e_3+e_4\}$\\%\cline{2-2}
&$\{e_1,\;e_2+e_3,\;e_4,\;e_2+e_3+e_4\}$\\%\cline{2-2}
&$\{e_1,\;e_2,\;e_3+e_4,\;e_2+e_3+e_4\}$\\%\cline{2-2}
&$\{e_1+e_2,\;e_3,\;e_4,\;e_2+e_3+e_4\}$\\%\cline{2-2}
&$\{e_2+e_3,\;e_4,\;e_4,\;e_1+e_2+e_3\}$\\%\cline{2-2} 
\hline
\end{tabular}
\end{center}
\end{minipage}
\begin{minipage}[t]{.5 \linewidth}
\begin{center}
\begin{tabular}[t]{|l|l|}\hline
$n$& Partitions of $\hr$ using $n$ vectors\\\hline\hline

\multirow{3}{*}{4}&$\{e_2,\;e_3+e_4,\;e_4,\;e_1+e_2+e_3\}$\\%\cline{2-2}
&$\{e_2,\;e_1+e_2,\;e_3+e_4,\;e_3+e_4\}$\\%\cline{2-2}
&$\{e_1+e_2,\;e_3+e_4,\;e_2,\;e_3+e_4\}$\\ \hline

\multirow{11}{*}{3}&$\{e_1,\;e_2,\;e_2+2e_3+2e_4\}$\\%\cline{2-2}
&$\{e_2,\;e_3,\;e_1+e_2+e_3+2e_4\}$\\%\cline{2-2}
&$\{e_1,\;e_2+e_3,\;e_2+e_3+4e_4\}$\\%\cline{2-2}
&$\{e_1+e_2,\;e_3,\;e_2+e_3+2e_4\}$\\%\cline{2-2}
&$\{e_1+e_2,\;e_2+e_3,\;e_3+2e_4\}$\\%\cline{2-2}
&$\{e_1+e_2+e_3,\;e_2,\;e_3+2e_4\}$\\%\cline{2-2}
&$\{e_2,\;e_3+e_4,\;e_1+e_2+e_3+e_4\}$\\%\cline{2-2}
&$\{e_2+e_3,\;e_4,\;e_1+e_2+e_3+e_4\}$\\%\cline{2-2}
&$\{e_1,\;e_2+e_3+e_4,\;e_2+e_3+e_4\}$\\%\cline{2-2}
&$\{e_1+e_2+e_3,\;e_4,\;e_2+e_3+e_4\}$\\%\cline{2-2}
&$\{e_1+e_2,\;e_3+e_4,\;e_2+e_3+e_4\}$\\\hline
\multirow{5}{*}{2}&$\{e_1+e_2+e_3+e_4,\;e_2+e_3+e_4\}$\\%\cline{2-2}
&$\{e_1+e_2+e_3,\;e_2+e_3+2e_4\}$\\%\cline{2-2}
&$\{e_1+e_2,\;e_2+2e_3+2e_4\}$\\%\cline{2-2}
&$\{e_2+e_3,\;e_1+e_2+e_3+2e_4\}$\\%\cline{2-2}
&$\{e_2,\;e_1+e_2+2e_3+2e_4\}$\\\hline
1&$\{e_1+2e_2+2e_3+2e_4\}$\\\hline
\end{tabular}
\end{center}
\end{minipage}

\caption{Partitions of $\hr$ in type $B_4$}\label{tab:example}
\end{table}

Let $P$ denote a partition of the vector $e_1+2e_2+\cdots+2e_{r-1}$.  If $P$ does not have a part containing  $2e_{r-1}$ as a summand, then $P$ contains exactly two parts each containing an $e_{r-1}$ summand.
Hence let $\A=e_i+\cdots+e_{r-1}$ and $\B=e_{j}+\cdots+e_{r-1}$ denote these two parts.

For any partition $P$ containing only nonhooked parts there are exactly 
four ways to extend these partitions $P$ of $e_1+2e_2+\cdots+2e_{r-1}$ to a partition of $e_1+2e_2+\cdots+2e_{r-1}+2e_{r}$, where 
the only parts that change are the ones containing $e_{r-1}$:
\begin{enumerate}[leftmargin=1.8cm]
 \item[$\EB(1)$:] replace $\A$ and $\B$ with $\bar{\A}=\A+e_r$ and $\bar{\B}=\B+e_r$ in $P$,
 \item[$\EB(2)$:] replace $\A$ with $\bar{\A}=\A+e_r$ in $P$, introduce the part $e_r$ into $P$, and leave $\B$ 
unchanged,
 \item[$\EB(3)$:] replace $\B$ with $\bar{\B}=\B+e_r$ in $P$, introduce the part $e_r$ to $P$, and leave $\A$ 
unchanged,
 \item[$\EB(4)$:] introduce the part $e_r$ twice in $P$, and leave $\A$ and $\B$ unchanged. 
\end{enumerate}

From the definition of a partition and the extensions above the following is immediate.
\begin{proposition}\label{proposition:overcountB}
Let $P$ be a partition of $e_1+2e_2+\cdots+2e_{r-1}$ that does not contain a hooked vector as a part, and for $1\leq \ell\leq 4$ let $P(\ell)$ be the result of applying extension $\EB(\ell)$ to $P$. Then for $\ell\neq k$, 
$P(\ell)=P(k)$ if and only if $\{\ell,k\}=\{2,3\}$. Furthermore, if $P(\ell)=P(k)$, then $\A=\B$.
\end{proposition}

It is important to note that if we began with two distinct partitions 
$P$ and $P'$ of the vector $e_1+2e_2+\cdots+2e_{r-1}$ which do not have a hooked vector as a part, then by the definition of the 
extensions $\EB(1)$, $\EB(2)$, $\EB(3)$, and $\EB(4)$, the extensions of $P$ and $P'$ do not yield the same final partition as 
these extensions only affect the parts that involve $e_r$.

\begin{definition}
Let $r\geq 2$, the let
\[\P_{B_r}^H(q)=c_0+c_1q+c_2q^2+\cdots+c_kq^k,\]
where $c_i$ is the number of partitions of $e_1+2e_2+\cdots+2e_r$ with $i$ parts where one part is a hooked vector.
Similarly, let
\[\P_{B_r}^{NH}(q)=c_0+c_1q+c_2q^2+\cdots+c_kq^k,\]
where $c_i$ is the number of ways to write $e_1+2e_2+\cdots+2e_r$ as a sum of exactly $i$ parts where no part is a hooked vector.
\end{definition}

\begin{lemma}\label{Lem:HookNoHook}
For $r\geq 3$, 
\[\P_{B_r}(q)=\P_{B_r}^{H}(q)+\P_{B_r}^{NH}(q).\]
\end{lemma}

\begin{proof}
This follows directly from Lemma~\ref{lem:typeB} and the fact that any partition of 
$e_1+2e_2+\cdots+2e_r$ has at most one part that is a hooked vector.
\end{proof}

\begin{proposition}\label{Prop:BNonHook}
For $r\geq 3$, the polynomials $\{\P_{B_i}^{NH}(q)\}$ satisfy the following recursion
\[\P_{B_r}^{NH} (q) = (1+q)^2 \P_{B_{r-1}}^{NH} (q) - q^3 \left( \sum_{i=1}^{r-2} \P_{B_i}^{NH}(q) \right ).\]
\end{proposition}
\begin{proof}
Recall that every partition $P$ of $e_1+2e_2+\cdots+2e_r$  with no hooked 
parts comes from an extension of a partition $P'$ of $e_1+2e_2+\cdots+2e_{r-1}$ with no hooked parts via at least one of the four extensions  
$\EB(1),$ $\EB(2),$ $\EB(3),$ and $\EB(4)$. These extensions respectively add zero, one, one, and two summands.

Hence, the polynomial whose coefficients encode the total of number of said extensions~is \[ (1+2q+q^2)\P_{B_{r-1}}^{NH} (q)=(1+q)^2 \P_{B_{r-1}}^{NH} (q) .  \] 
However, by Proposition \ref{proposition:overcountB},  $ (1+q)^2  \P_{B_{r-1}}^{NH} (q) $ double 
counts the contribution of partitions $P$ of $e_1+2e_2+\cdots+2e_r$ obtained from extensions $\EB(2)$ and $\EB(3)$ of partitions 
$P'$ of $e_1+2e_2+\cdots+2e_{r-1}$ that contain two equal parts $\A$ and $\B$ with \[ \A=\B=e_{i+1} + \cdots + e_{r-1} \] for 
some $i$ with $1\leq i \leq r-2$. Consider these partitions for a fixed $i$.  By removing $\A$ and $\B$ from such a partition $P'$, we see that these partitions are in bijection with the set of partitions of $e_1+2e_2+\cdots+2e_i$.  Thus, the polynomial whose coefficients encode the total number of such partitions is given by $q^3\P_{B_i}^{NH}(q)$, where we multiply 
 by $q^3$ because $\bar{P}(2)$ and $\bar{P}(3)$ each have three more parts than $P'$. 
By ranging over all possible $i$ we have that the polynomial encoding the double counted partitions is 
 \[ q^3 \left( \sum_{i=1}^{r-2} \P_{B_i}^{NH}(q) \right ) .\] 
 Thus \[\P_{B_r}^{NH} (q) = (1+q)^2 \P_{B_{r-1}}^{NH} (q) - q^3 \left( \sum_{i=1}^{r-2} \P_{B_i}^{NH}(q) \right ). \qedhere \]
\end{proof}

\begin{proposition}\label{Prop:BNonHookGenFun}
The closed form for the generating function $\textstyle\sum_{r \geq 1} \P_{B_r}^{NH}(q)x^r$ is
\[
\frac{qx+(-2q-q^2)x^2+(q+q^2)x^3}{1-(2+2q+q^2)x+(1+2q+q^2+q^3)x^2}.
\]
\end{proposition}

\begin{proof} 
For simplicity, let $S(q,x)=\sum_{r \geq 1} \P_{B_r}(q)x^r$.  Then 
\begin{align*}
S(q,x) &= \P_{B_1}^{NH}(q)x + \P_{B_2}^{NH}(q)x^2 + \sum_{r \geq 3} \P_{B_r}^{NH}(q)x^r \\
&= qx + (q^3+q^2)x^2 + \sum_{r \geq 3} \left[ (1+q)^2 \P_{B_{r-1}}^{NH} (q) - q^3 \left( \sum_{i=1}^{r-2} \P_{B_i}^{NH}(q) \right ) \right] x^r \\
&= qx + (q^3+q^2)x^2 + x (1+q)^2 \left( \sum_{r \geq 3} \P_{B_{r-1}}^{NH}(q)x^{r-1} \right) - q^3x^2 \sum_{r \geq 3} \left( \sum_{i=1}^{r-2} \P_{B_i}^{NH}(q) \right) x^{r-2} \\
&= qx + (q^3+q^2)x^2 + x(1+q)^2 \left( S(q,x) - qx \right) - \frac{q^3x^2}{1-x}S(q,x).
\end{align*}
The second equality given above comes from Proposition~\ref{Prop:BNonHook}. Hence,
\[
S(q,x) \cdot \left( 1 - x(1+q)^2 + \dfrac{q^3x^2}{1-x} \right) = qx + (q^3+q^2)x^2 - qx^2(1+q)^2,
\]
so
\[
S(q,x) = \frac{qx+(-2q-q^2)x^2+(q+q^2)x^3}{1-(2+2q+q^2)x+(1+2q+q^2+q^3)x^2}. \qedhere
\]
\end{proof}

\begin{proposition}\label{Prop:BHook}
For $r \geq 3$, the polynomials $\{\P_{B_i}^H(q)\},\{\P_{B_i}^{NH}(q)\}$ satisfy the following recurrence
\[
\P_{B_r}^{H} (q) = q + 2 \left( \sum_{i=2}^{r-1} \P_{B_i}^{NH} (q) \right) - q^2 \left(  \sum_{i=1}^{r-2} \sum_{j=1}^i \P_{B_j}^{NH} (q) \right) .
\]
\end{proposition}

\begin{proof}
Let $i \in \{2,3\ldots,r-1\}$.  Let $P$ be a partition of $e_1+2e_2+\cdots+2e_i$ that does not have any hooked part.  
Then $P$ has two parts that contain $e_i$ as a summand.  Let these two parts be denoted $\A=e_j+\cdots+e_i$ and 
$\B=e_k+\cdots+e_i$.  We can extend $P$ to a partition of $B_r$ that has a hooked part: either replace $\A$ by 
$e_j+\cdots+e_i+2e_{i+1} + \cdots + 2e_{r}$ or replace $\B$ with $e_k+\cdots+e_i+2e_{i+1} + \cdots + 2e_{r}$.  Ignoring the full 
partition $e_1+2e_2+\cdots+2e_r$, every partition of $e_1+2e_2+\cdots+2e_r$ that has a hooked part can be constructed by extending 
a partition $P$ of $e_1+2e_2+\cdots+2e_i$ for some $i$ using the aforementioned process.  Indeed, suppose $Q$ is a partition of 
$e_1+2e_2+\cdots+2e_r$ that has a hooked part, say $e_j+e_{j+1}+\cdots+e_{i}+2e_{i+1}+\cdots+2e_{r}$.  
Then the partition $P$ that 
has all the same parts as $Q$ but replaces $e_j+e_{j+1}+\cdots+e_{i}+2e_{i+1}+\cdots+2e_{r}$ with $e_j+e_{j+1}+\cdots+e_{i}$ is a 
partition of $e_1+2e_2+\cdots+2e_i$ that extends to $Q$.  Consequently, we have
\[
\P_{B_r}^{H} (q) = q + 2 \left( \sum_{i=2}^{r-1} \P_{B_i}^{NH} (q) \right) - G(q)
\]
where $G(q)$ subtracts the contribution of partitions of  $e_1+2e_2+\cdots+2e_r$ that can arise
from partitions of $e_1+2e_2,\ e_1+2e_2+2e_3,\ldots,e_1+2e_2+\cdots+2e_{r-1}$ (that do not have any hooked part) via the above extension in 
multiple ways. 

To compute $G(q)$, we suppose $P$ and $P'$ are partitions of $e_1+2e_2+\cdots+2e_i$ 
and $e_1+2e_2+\cdots+2e_{i'}$ respectively for which the above extension leads to the same partition $Q$ of 
$e_1+2e_2+\cdots+2e_r$.  Only one part of the partition achieved after extending $P$ (similarly $P'$) to $Q$ contains a hooked 
part, and it is of the form $e_k+\cdots+e_{i}+2e_{i+1}+\cdots+2e_r$ (similarly $e_{k'}+\cdots+e_{i'}+2e_{i'+1}+\cdots+2e_r$).  
Since $P$ and $P'$ both extend to $Q$, this implies 
\[
e_k+\cdots+e_{i}+2e_{i+1}+\cdots+2e_r = e_{k'}+\cdots+e_{i'}+2e_{i'+1}+\cdots+2e_r
\]
and hence $i=i'$ and $e_k+\cdots+e_{i} = e_{k'}+\cdots+e_{i'}$. Since all other parts of $P$ and $P'$ are 
the same as the parts of $Q$ besides the part containing a $2e_r$ summand, we deduce $P=P'$.

Thus, in order to determine $G(q)$, we need to determine when 
applying the two different aforementioned extensions to a partition $P$ (not containing a hooked part)  of $e_1+2e_2+\cdots+2e_i$ for some $i \in \{2,\ldots,r-1\}$ can result in the same partition.  Let $P$ be a partition of $e_1+2e_2+\cdots+2e_i$ not containing a hooked 
part.  Then $P$ has exactly two parts containing $e_i$ as a summand.  Call these $\A=e_k+\cdots+e_i$ and $\B=e_{k'}+\cdots+e_i$.  
When applying our extension, these are replaced by $e_k+\cdots+e_i+2e_{i+1} + \cdots + 2e_{r}$ and $e_{k'}+\cdots+e_i+2e_{i+1} + 
\cdots + 2e_{r}$ respectively, and all other summands remain the same, so the two extensions are equal if and only if $\A=\B$ 
(i.e. $k=k'$).  

We can now compute $G(q)$.  Fix $i \in \{2,\ldots,r-1\}$.  
We determine the contribution of the set of partitions $P$ of $e_1+2e_2+\cdots+2e_i$ (not containing a hooked part) to $G(q)$.  
From the previous paragraph, we obtain a contribution to $G(q)$ for every such partition $P$ in which its two parts $\A,\B$ 
containing $e_i$ are the same, say $\A=\B=e_{j+1}+\cdots+e_i$.  The remainder of the partition $P$ can then range over any 
partition of $e_1+2e_2+\cdots+2e_j$ with $1 \leq j < i$ not containing a hooked part.  For each such partition of 
$e_1+2e_2+\cdots+2e_j$, the partition $P$ (and hence its extension) has two more parts than it, accounting for $\A$ and $\B$.  
Thus, the combined contribution to $G(q)$ arising from partitions of $e_1+2e_2+\cdots+2e_i$ is given by \[\sum_{j=1}^{i-1} q^2 
\P_{B_j}^{NH}(q).\]  Ranging over all $i$ gives us
\[
G(q)=\sum_{i=2}^{r-1} \sum_{j=1}^{i-1} q^2 \P_{B_j}^{NH}(q) = \sum_{i=1}^{r-2} \sum_{j=1}^{i} q^2 \P_{B_j}^{NH}(q) . \qedhere
\]
\end{proof}

\begin{theorem}\label{Bgenfun} The closed formula for the generating function $\textstyle\sum_{r\geq 1}\P_{B_{r}}(q)x^r$ is given by
\[\displaystyle\sum_{r\geq 1}\P_{B_r}(q)x^r=\frac{qx+(-q-q^2)x^2+q^2x^3}{1-(2+2q+q^2)x+(1+2q+q^2+q^3)x^2}          .\]
\end{theorem}

\begin{proof} 
For simplicity, let $T(q,x)=\textstyle\sum_{r \geq 1} \P_{B_r}^{H}(q)x^r$, and $S(q,x)=\textstyle\sum_{r \geq 1} \P_{B_r}^{NH}(q)x^r$.  Then 
\begin{align*}
T(q,x) &= \P_{B_1}^{H}(q)x + \P_{B_2}^{H}(q)x^2 + \sum_{r \geq 3} \P_{B_r}^{H}(q)x^r \\
&= qx^2 + \sum_{r \geq 3} \left( q + 2 \left( \sum_{i=2}^{r-1} \P_{B_i}^{NH} (q) \right) - q^2 \left(  \sum_{i=1}^{r-2} \sum_{j=1}^i \P_{B_j}^{NH} (q) \right) \right) x^r \\
&= qx^2 + \sum_{r \geq 3} qx^r + \sum_{r \geq 3}  2 \left( \sum_{i=2}^{r-1} \P_{B_i}^{NH} (q) \right)x^r - q^2 \sum_{r \geq 3} \left(  \sum_{i=1}^{r-2} \sum_{j=1}^i \P_{B_j}^{NH} (q) \right)x^r \\
&= qx^2 + \dfrac{qx^3}{1-x} + \dfrac{2x}{1-x} (S(q,x)-qx) - \dfrac{q^2x^2}{(1-x)^2} S(q,x).
\end{align*}
The second equality follows from Proposition~\ref{Prop:BHook}.    Hence,
\begin{align*}
\sum_{r \geq 1} \P_{B_r}(q)x^r &= S(q,x)+T(q,x) \\&= S(q,x) + qx^2 + \dfrac{qx^3}{1-x} - \dfrac{2qx^2}{1-x}+ \dfrac{2x}{1-x} S(q,x) - \dfrac{q^2x^2}{(1-x)^2} S(q,x),
\end{align*}
the first equality following from Lemma~\ref{Lem:HookNoHook}.  It follows then that
\[
\sum_{r \geq 1} \P_{B_r}(q)x^r = \frac{qx+(-q-q^2)x^2+q^2x^3}{1-(2+2q+q^2)x+(1+2q+q^2+q^3)x^2}. \qedhere
\]
\end{proof}

\section{Type $C$} \label{type C}

Let $\P_{C_r}(q):=\wp_q(\hr)$, where $\hr=2\alpha_1+2\alpha_2+\cdots+2\alpha_{r-1}+\alpha_r$ is the highest root of the Lie algebra of type $C_r$.  By Lemma~\ref{lem:TypeC}, we can recover $\P_{C_r}(q)$ by determining the number of partitions of $e_1+2e_2+\cdots+2e_{r}$ into $k$ parts where the parts are as specified in Lemma~\ref{lem:TypeC}.  In this light, throughout this section, a \emph{partition} of the vector $e_1+2e_2+\cdots+2e_{\ell}$ in $\CC^{\infty}$ where $2 \leq \ell \leq r$ is a nonnegative integer combination of $k$ (not necessarily distinct) vectors from the following~set:
 \begin{itemize}
\item Nonhooked vectors (parts) of the form $e_i +e_{i+1}+\cdots + e_j$ with $1 \leq i \leq j \leq \ell$. 
\item Hooked vectors (parts) of the form $e_1 + 2e_{2} + \cdots + 2e_{i} +e_{i+1} + \cdots + e_j$ with $1 < i < j \leq \ell$.  
\end{itemize}

Let $P$ be a partition of $\Psi(\hr) = e_1+2e_2+2e_3+\cdots+2e_{r-1} \in \CC^\infty$.
If $P$ has more than one part, then exactly two of the parts of $P$ must contain $e_{r-1}$ as a summand. 
Let $\A = e_i + \cdots +e_{r-1} $ and $\B = e_{j} + \cdots + e_{r-1}$ be these two parts.  For any partition $P$ besides the 
partition $\{e_1+2e_2+\cdots+2e_{r-1}\}$ there are exactly four ways to extend a partition $P$ of $e_1+2e_2+\cdots+2e_{r-1}$ to a 
partition of $e_1+2e_2+\cdots+2e_{r-1}+2e_r$, where the parts that do not contain $e_{r-1}$ remain the same:
\begin{enumerate}[leftmargin=1.8cm]
 \item [$\E(1)$:] introduce two $e_r$ parts to $P$ to get $P \cup \{ e_r,e_r \}$,
 \item [$\E(2)$:] replace $\A$ by $\bar{\A} = e_i + \cdots +e_{r-1}+ e_r$, introduce the part $e_r$ to $P$, and leave $\B$ 
unchanged,
 \item [$\E(3)$:] replace $\B$ by $\bar{\B}  = e_{j} + \cdots + e_{r-1}+ e_r$, introduce the part  $e_r$ to $P$, and leave $\A$ 
unchanged,
 \item [$\E(4)$:] replace both $\A$ and $\B$ by $\bar{\A} = e_i + \cdots +e_{r-1}+ e_r$ and  $\bar{\B}  = e_{j} + \cdots + 
e_{r-1}+ e_r$ respectively.
\end{enumerate}
We remark that the only remaining partitions $P$ of $e_1+2e_2+2e_3+\cdots+2e_r$ in $\mathbb{C}^{\infty}$ which are not formed 
from the above extensions are the following three partitions:
\begin{align}
P_1&=\{e_1+2e_2+\cdots+2e_r\}\label{partition1}\\
P_2&=\{e_1+2e_2+\cdots+2e_{r-1}+e_r,e_r\}\label{partition2}\\
P_3&= \{e_1+2e_2+\cdots+2e_{r-1},e_r,e_r\}\label{partition3}
\end{align}
which contain one, two, and three parts respectively. 
From the definition of a partition and the extensions above the following is immediate.
\begin{proposition}\label{proposition:overcount}
Let $P$ be a partition of $\Psi(\hr) = e_1+2e_2+\cdots+2e_{r-1} \in \CC^\infty$ distinct from 
the partition with only one part $\{e_1+2e_2+\cdots+2e_{r-1}\}$ and for $1\leq \ell\leq 4$ let $P(\ell)$ be the result of applying 
extension $\E(\ell)$ to $P$. Then for $\ell\neq k$, $P(\ell)=P(k)$ if and only if $\{\ell,k\}=\{2,3\}$. Furthermore, if 
$P(\ell)=P(k)$, then $\A=\B$.
\end{proposition}

Note that the extensions $\E(1)$, $\E(2)$, $E_C(3)$, and $\E(4)$ only affect parts that contain $e_{r-1}$ as a summand, hence if you start with two distinct partitions $P$ and $P'$ of $e_1+2e_2+2e_3+\cdots+2e_{r-1}$ in $\mathbb{C}^{\infty}$ then the extensions $P(\ell)\neq P'(j)$ for any $1\leq \ell,j\leq 4$.

\begin{theorem}\label{Cthm}
If $r\geq 3$, then the polynomials $\{\P_{C_i}(q)\}$ satisfy the following recurrence
\[\P_{C_r} (q) = (1+q)^2( \P_{C_{r-1}} (q)-q) - q^3 \left( \sum_{i=1}^{r-2} \P_{C_i}(q) \right )+(q+q^2+q^3) .\] 
\end{theorem}

\begin{proof}
Recall that every partition $P$ (except for the three partitions $P_1$, $P_2$ and $P_3$ listed in Equations \eqref{partition1}, \eqref{partition2}, and \eqref{partition3}, respectively) of the vector $e_1+2e_2+\cdots+2e_{r-1}+2e_r$ comes from an extension of a partition $P'$ of $e_1+2e_2+\cdots+2e_{r-1}$ except the full partition $\{e_1+2e_2+\cdots+2e_{r-1}\}$ via the four extensions $\E(1),$ $\E(2),$ $\E(3),$ and $\E(4)$. These extensions respectively add two, one, one, and zero parts.  The polynomial whose coefficients encode the total of number of said extensions is \[ (1+2q+q^2) (\P_{C_{r-1}} (q)-q)= (1+q)^2( \P_{C_{r-1}} (q)-q).  \] 
However, by Proposition~\ref{proposition:overcount}, $ (1+q)^2 (\P_{C_{r-1}} (q)-q)$ double counts the contribution of the 
partitions $P$ of $e_1+2e_2+\cdots+2e_{r-1}+2e_r$ obtained from extensions $\E(2)$ and $\E(3)$ of partitions $P'$ of 
$e_1+2e_2+\cdots+2e_{r-1}$ which contain two equal parts 
$\A$ and $\B$ such that \[ \A=\B=e_{i+1} + \cdots + e_{r-1} \] for some $i$ with $1\leq i \leq r-2$. Consider such partitions $P'$ for a fixed $i$.
The set  of such partitions correspond bijectively with the partitions of $e_1+2e_2+\cdots+2e_i$. 
Thus the polynomial encoding such double counted partitions is given by $q^3\P_{C_i}(q)$, where we multiply by $q^3$ because $P'(2)$ and 
$P'(3)$ each have three more parts than $P'$. 
By ranging over all possible $i$ we have that the total number of double counted partitions is 
 \[ q^3 \left( \sum_{i=1}^{r-2} \P_{C_i}(q) \right ) .\] 
 Thus \[\P_{C_r} (q) = (1+q)^2( \P_{C_{r-1}} (q)-q) - q^3 \left( \sum_{i=1}^{r-2} \P_{C_i}(q) \right )+(q+q^2+q^3) ,\]
 where the term $(q+q^2+q^3)$ is the contribution from the three partitions $P_1$, $P_2$, and $P_3$ given in Equations 
\eqref{partition1}, 
\eqref{partition2}, and \eqref{partition3}.
\end{proof}

\begin{theorem}\label{Cgenfun}  The closed formula for the generating function $\textstyle\sum_{k\geq 1}\P_{C_{k}}(q)x^k$ is given by
\[\displaystyle\sum_{k\geq 1}\P_{C_k}(q)x^k=\frac{qx+(-q^2-q)x^2}{1-(2+2q+q^2)x+(1+2q+q^2+q^3)x^2}.\]
\end{theorem}

\begin{proof}
Let \begin{align}\displaystyle\sum_{k\geq 1}f_k(q)x^k&=\frac{qx+(-q^2-q)x^2}{1-(2+2q+q^2)x+(1+2q+q^2+q^3)x^2}\label{fks}.\end{align} We will show that $f_k(q)=\P_{C_k}(q)$ for all $k\geq 1$.
We proceed by induction, and we observe that 
\begin{align*}
f_1(q)&=\P_{C_1}(q)=q,\\
f_2(q)&=\P_{C_2}(q)=q(q^2+q+1), \mbox{and}\\ 
f_3(q)&=\P_{C_3}(q)=q(q^4+2q^3+4q^2+2q+1).
\end{align*}
By induction we can assume that $f_{k-1}(q)=\P_{C_{k-1}}(q)$
 and $f_{k-2}(q)=\P_{C_{k-2}}(q)$. From the rational expression of the generating formula given in Equation \eqref{fks} we know that 
 \begin{align*}
f_k(q)&=(2+2q+q^2)f_{k-1}(q)-(1+2q+q^2+q^3)f_{k-2}(q)\\
 &=  (1+q)^2f_{k-1}(q)-q^3f_{k-2}(q)+f_{k-1}(q)-(1+q)^2
f_{k-2}(q), \text{and by induction hypothesis}\\
&=(1+q)^2\P_{C_{k-1}}(q)-q^3\P_{C_{k-2}}(q)+\P_{C_{k-1}}(q)-(1+q)^2
\P_{C_{k-2}}(q).
\end{align*}
Using Theorem \ref{Cthm} we note that
\[(1+q)^2\P_{C_{k-1}}(q)-q^3\P_{C_{k-2}}(q)=q^2+\P_{C_k}(q)+q^3\left(\displaystyle\sum_{i= 1}^{k-3}\P_{C_i}(q)\right)\]
and
\[\P_{C_{k-1}}(q)-(1+q)^2
\P_{C_{k-2}}(q)=-q^2-q^3\left(\displaystyle\sum_{i= 1}^{k-3}\P_{C_i}(q)\right).\]
Adding the last two equalities yields the desired result. 
\end{proof}

\section{Type $D$} \label{type D}

Let $\P_{D_r}(q):=\wp_q(\hr)$, where $\hr=\alpha_1+2\alpha_2+\cdots+2\alpha_{r-2}+\alpha_{r-1}+\alpha_r$ is the highest root of the Lie algebra of type $D_r$. By Lemma~\ref{lem:TypeD}, we can 
recover $\P_{D_r}(q)$ by determining the number of partitions of $e_1+2e_2+\cdots+2e_{r-2}+e_{r-1}+e_r$ into $k$ parts where the 
parts are as specified in Lemma~\ref{lem:TypeD}.  In this light, throughout this section, for any $\rt \geq 5$ we refer to a 
\emph{partition} of $e_1+2e_2+\cdots+2e_{\rt-2}+e_{\rt-1}+e_{\rt} \in \mathbb{C}^{\infty}$, or equivalently a partition of the highest root in type $D_{\ell}$, as a nonnegative integer combination of $k$ (not necessarily distinct) vectors from the following set:

\begin{itemize}
\item Nonhooked vectors (parts) of the form $e_i +e_{i+1}+\cdots + e_{j}$ with $1 \leq i \leq j \leq \rt$. 
\item Nonhooked vectors (parts) of the form $e_i +e_{i+1}+\cdots + e_{\rt-2}+e_{\rt}$ with $1 \leq i  \leq \rt-2$. 
\item Hooked vectors (parts) of the the form  \[e_i+e_{i+1}+ \cdots + e_{j-1}+2 e_j+2 e_{j+1}+\cdots +2 e_{\rt-2}+e_{\rt-1}+e_{\rt} \] 
with $  1\leq i  < j \leq \rt-2$.
\end{itemize}

Throughout our proof, we will relate the polynomials $\{\P_{D_i}(q)\}$ to the polynomials $\{\P_{B_i}^H(q)\}$ and 
$\{\P_{B_i}^{NH}(q)\}$ from Section \ref{type B}.  As such, for $r \geq 3$, we refer to a \emph{partition} of $e_1+2e_2+2e_3+\cdots + 2e_{\ell} \in 
\mathbb{C}^{\infty}$, or equivalently a partition of the highest root in type $B_{\rt}$, as a nonnegative integer combination of $k$ (not necessarily distinct) vectors from the following set:

\begin{itemize}
\item Nonhooked vectors (parts) of the form $e_i +e_{i+1}+\cdots + e_j$ with $1 \leq i \leq j \leq \rt$.  
\item Hooked vectors (parts) of the the form  $e_i + e_{i+1}+\cdots + e_{j-1} +2e_{j} + 2e_{j+1} +\cdots + 2e_{\rt}$ with $1 \leq 
i < j \leq \rt$.
\end{itemize}
In order to determine the generating function for the sequence $\{P_{D_r}(q)\}_{r \geq 4}$, we explictly relate this sequence to the polynomials for the Lie algebras of type $B$. Before stating the first result we provide the following.
\begin{definition}
Let $r\geq 2$, the let
\[\P_{D_r}^H(q)=c_0+c_1q+c_2q^2+\cdots+c_kq^k,\]
where $c_i$ is the number of partitions of $e_1+2e_2+\cdots+2e_{r-2}+e_{r-1}+e_r$ with $i$ parts where one part is a hooked vector (of type $D_r$).
Similarly, let
\[\P_{D_r}^{NH}(q)=c_0+c_1q+c_2q^2+\cdots+c_kq^k,\]
where $c_i$ is the number of ways to write $e_1+2e_2+\cdots+2e_{r-2}+e_{r-1}+e_r$ as a sum of exactly $i$ parts where no parts are hooked vectors (of type $D_r$).
\end{definition}

\begin{theorem}\label{thm:Dthm}
For $r \geq 2$,
\[
\P_{D_{r+2}}(q) = \P_{B_{r}}^H(q) + q^2 \P_{B_{r}}^{NH}(q) + (2q+2) \left( 2 \P_{B_{r}}^{NH}(q) - q^2 \sum_{i=1}^{r-1} 
\P_{B_{i}}^{NH}(q) \right).
\]
\end{theorem}
\begin{proof}
We first observe there is a bijection between partitions of $e_1+2e_2+\cdots+2e_{r}+e_{r+1}+e_{r+2}$ that have a hooked part and 
partitions of $e_1+2e_2+2e_3+\cdots + 2e_r$ that have a hooked part that preserves the number of parts in each partition.  Indeed, 
this bijection takes any partition of $e_1+2e_2+\cdots+2e_{r}+e_{r+1}+e_{r+2}$ and removes the $e_{r+1}+e_{r+2}$ from its hooked 
part.  From this we deduce $\P_{D_{r+2}}^H(q)=\P_{B_{r}}^H(q)$.  It therefore remains to show that
\[
\P_{D_{r+2}}^{NH}(q) = q^2 \P_{B_{r}}^{NH}(q) + (2q+2) \left( 2 \P_{B_{r}}^{NH}(q) - q^2 \sum_{i=1}^{r-1} \P_{B_{i}}^{NH}(q) 
\right).
\]

We split this into five cases depending on the partitions of $e_1+2e_2+\cdots+2e_{r}+e_{r+1}+e_{r+2}$ and the parts which contain the summands $e_{r+1}$ and $e_{r+2}$.\\

\noindent
{\bf Case 1}: This case considers partitions of $e_1+2e_2+\cdots+2e_{r}+e_{r+1}+e_{r+2}$ containing $e_{r+1}$ and $e_{r+2}$ as parts. The polynomial encoding the count for all such 
partitions is $q^2\P_{B_r}^{NH}(q)$ because each of these is uniquely obtained by introducing the parts $\{e_{r+1},e_{r+2}\}$ to a 
partition of $e_1+2e_2+2e_3+\cdots + 2e_r$ (in type $B_r$) that itself has no hooked part.\\

\noindent
{\bf Case 2}: This case considers partitions of $e_1+2e_2+\cdots+2e_{r}+e_{r+1}+e_{r+2}$ containing $e_{r+2}$ as a part, but in which $e_{r+1}$ is not a part. Every partition of 
$e_1+2e_2+2e_3+\cdots + 2e_r$ without a hooked part can be extended in two ways to get such a partition by adding $e_{r+1}$ to a 
summand involving $e_r$.  On the level of polynomials, this gives $2q \P_{B_r}^{NH}(q)$ (the $q$ comes from 
introducing the lone $e_{r+2}$ 
part to the partition of $e_1+2e_2+2e_3+\cdots + 2e_r$).  However, this double counts the contributions of 
partitions of  $e_1+2e_2+2e_3+\cdots + 2e_r$ whose two parts 
containing $e_r$ are the same.  The over count is given by $q \left( q^2 \textstyle\sum_{i=1}^{r-1} \P_{B_{i}}^{NH}(q) \right)$, the $q$ for 
the lone $e_{r+2}$ part, and the $q^2$ for the two summands containing $e_r$ extended from a partition of $e_1+2e_2+2e_3+\cdots 
+ 2e_i$ for some $i \in \{1,2\ldots,r-1\}$.  On the level of polynomials this gives us
\[
q \left( 2 \P_{B_{r}}^{NH}(q) - q^2 \sum_{i=1}^{r-1} \P_{B_{i}}^{NH}(q) \right).
\]

\noindent
{\bf Case 3}: This case considers partitions of $e_1+2e_2+\cdots+2e_{r}+e_{r+1}+e_{r+2}$ containing $e_{r+1}$ as a part, but in which $e_{r+2}$ is not a part.
 This argument follows 
directly from the argument in the previous case and gives us
\[
q \left( 2 \P_{B_{r}}^{NH}(q) - q^2 \sum_{i=1}^{r-1} \P_{B_{i}}^{NH}(q) \right).
\]

\noindent
{\bf Case 4}: This case considers the partitions of $e_1+2e_2+\cdots+2e_{r}+e_{r+1}+e_{r+2}$ with $e_{r+1}$ and $e_{r+2}$ as summands in different parts. Any 
partition of $e_1+2e_2+2e_3+\cdots + 2e_r$ that does not have a hooked part can be extended to such a partition in two ways: by 
adding $e_{r+1}$ to one part containing $e_r$ and adding $e_{r+2}$ to the other part containing $e_r$.  This does not change the 
number of parts so on the level of polynomials we get $2 \P_{B_{r}}^{NH}(q)$, but we must subtract double counts which come from 
those partitions of $$e_1+2e_2+2e_3+\cdots + 2e_r$$ whose two parts containing $e_r$ are the same.  By a similar argument to the 
previous case, this double count is accounted for by $q^2 \textstyle\sum_{i=1}^{r-1} \P_{B_{i}}^{NH}(q)$. Hence the polynomial encoding the count for such 
partitions is
\[
 2 \P_{B_{r}}^{NH}(q) - q^2 \sum_{i=1}^{r-1} \P_{B_{i}}^{NH}(q).
\]

\noindent
{\bf Case 5}: This case considers the remaining partitions of $e_1+2e_2+\cdots+2e_{r}+e_{r+1}+e_{r+2}$: those with $e_{r+1}$ and $e_{r+2}$ as summands of the same part. Any 
partition of $e_1+2e_2+2e_3+\cdots + 2e_r$ not containing a hooked part can be extended to such a partition by adding 
$e_{r+1}+e_{r+2}$ to a summand containing $e_r$.  This does not change the number of parts so the total we get on the level of 
polynomials is $2 \P_{B_{r}}^{NH}(q)$.  But we must subtract double counts which come from those partitions in 
$e_1+2e_2+2e_3+\cdots + 2e_r$ whose two summands containing $e_r$ are the same.  This is accounted for by $q^2 \textstyle\sum_{i=1}^{r-1} 
\P_{B_{i}}^{NH}(q)$.  So the polynomial encoding the count for such partitions is given by
\[
 2 \P_{B_{r}}^{NH}(q) - q^2 \sum_{i=1}^{r-1} \P_{B_{i}}^{NH}(q).
\]

Adding these five cases yields the desired result.
\end{proof}

\begin{theorem} \label{thm:DGenFun}
The closed form for the generating series $\textstyle\sum_{r \geq 4} \P_{D_r}(q)x^r$ is 
\[
\sum_{r \geq 4} \P_{D_r}(q)x^r = \dfrac{(q+4q^2+6q^3+3q^4+q^5)x^4 - 
(q+4q^2+6q^3+5q^4+3q^5+q^6)x^5}{1-(2+2q+q^2)x+(1+2q+q^2+q^3)x^2}.
\]
\end{theorem}
\begin{proof}Observe that
\begin{align*}
\hspace{-7mm}\sum_{r \geq 4} \P_{D_r}(q)x^r &= \sum_{r \geq 2} \P_{D_{r+2}}(q)x^{r+2} \\
&=  \sum_{r \geq 2} \left( \P_{B_{r}}^H(q) + q^2 \P_{B_{r}}^{NH}(q) + (2q+2) \left( 2 \P_{B_{r}}^{NH}(q) - q^2 \sum_{i=1}^{r-1} 
\P_{B_{i}}^{NH}(q) \right) \right) x^{r+2} \\
&= \sum_{r \geq 2} \P_{B_{r}}^H(q) x^{r+2} + q^2 \sum_{r \geq 2} \P_{B_{r}}^{NH}(q) x^{r+2} + (4q+4)\sum_{r \geq 
2}\P_{B_{r}}^{NH}(q) x^{r+2} \\
&\;\;\;\;-q^2(2q+2) \sum_{r \geq 2} \left( \sum_{i=1}^{r-1} \P_{B_{i}}^{NH}(q) \right) x^{r+2} \\
&= x^2 \sum_{r \geq 2} \P_{B_{r}}^H(q) x^{r} + q^2x^2 \sum_{r \geq 2} \P_{B_{r}}^{NH}(q) x^{r} + (4q+4)x^2 \sum_{r \geq 
2}\P_{B_{r}}^{NH}(q) x^{r} \\
&\;\;\;\;- \dfrac{q^2(2q+2)x^3}{1-x} \sum_{r \geq 1} \P_{B_{r}}^{NH}(q) x^r.
\end{align*}
Now for simplicity, let 
\[
S(q,x) = \sum_{r \geq 1} \P_{B_{r}}^H(q) x^{r}, \ \ T(q,x) = \sum_{r \geq 1} \P_{B_{r}}^{NH}(q) x^{r}.
\]
Directly from Proposition~\ref{Prop:BNonHookGenFun}, we have 
\[
T(q,x)=\frac{qx+(-2q-q^2)x^2+(q+q^2)x^3}{1-(2+2q+q^2)x+(1+2q+q^2+q^3)x^2}
\]
and Lemma~\ref{Lem:HookNoHook} together with Theorem~\ref{Bgenfun} implies
\begin{align*}
\hspace{-3mm}S(q,x) &= \frac{qx+(-q-q^2)x^2+q^2x^3}{1-(2+2q+q^2)x+(1+2q+q^2+q^3)x^2} - T(q,x)\\
&= \frac{qx+(-q-q^2)x^2+q^2x^3}{1-(2+2q+q^2)x+(1+2q+q^2+q^3)x^2} - 
\frac{qx+(-2q-q^2)x^2+(q+q^2)x^3}{1-(2+2q+q^2)x+(1+2q+q^2+q^3)x^2} \\
&= \frac{qx^2-qx^3}{1-(2+2q+q^2)x+(1+2q+q^2+q^3)x^2}.
\end{align*}
Thus,
\begin{align*}
\hspace{-2mm}\sum_{r \geq 4} \P_{D_r}(q)x^r &= x^2 \sum_{r \geq 2} \P_{B_{r}}^H(q) x^{r} + q^2x^2 \sum_{r \geq 2} \P_{B_{r}}^{NH}(q) x^{r} + 
(4q+4)x^2 \sum_{r \geq 2}\P_{B_{r}}^{NH}(q) x^{r} \\
&\;\;\;\;- \dfrac{q^2(2q+2)x^3}{1-x} \sum_{r \geq 1} \P_{B_{r}}^{NH}(q) x^r \\
&= x^2 S(q,x) + q^2x^2(T(q,x)-qx) + (4q+4)x^2(T(q,x)-qx) - \dfrac{q^2(2q+2)x^3}{1-x} T(q,x) \\
&= \dfrac{(q+4q^2+6q^3+3q^4+q^5)x^4 - (q+4q^2+6q^3+5q^4+3q^5+q^6)x^5}{1-(2+2q+q^2)x+(1+2q+q^2+q^3)x^2}. \qedhere
\end{align*}
\end{proof}

\section*{Acknowledgments}
The first and third author would like to thank the American Mathematical Society's Mathematics Research Communities program on Algebraic and Geometric Methods in Applied Discrete Mathematics for initiating their collaboration. The first author gratefully acknowledges travel support from the Faculty Development Research Fund at the United States Military Academy.
\appendix
\section{Appendix: Explicit Formulas}\label{appendix}
\begin{proof}[Proof of Corollary~\ref{corclosed}]
Prior to proceeding to each Lie type individually we recall the generating functions for $\textstyle\sum_{r \geq 1} \P_{B_r}(q)x^r$, $\textstyle\sum_{r \geq 1} \P_{C_r}(q)x^r$ and $\textstyle\sum_{r \geq 4} \P_{D_r}(q)x^r$ satisfy the rational expressions
\begin{align}
\sum_{r \geq 1} \P_{B_r}(q)x^r&=
\frac{qx+(-q-q^2)x^2+q^2x^3}{1-(2+2q+q^2)x+(1+2q+q^2+q^3)x^2}\label{typebgen}\\
\sum_{r \geq 1} \P_{C_r}(q)x^r&=\frac{qx+(-q-q^2)x^2}{1-(2+2q+q^2)x+(1+2q+q^2+q^3)x^2}\label{typecgen}\\
\sum_{r \geq 4} \P_{D_r}(q)x^r&= \dfrac{(q+4q^2+6q^3+3q^4+q^5)x^4 - (q+4q^2+6q^3+5q^4+3q^5+q^6)x^5}{1-(2+2q+q^2)x+(1+2q+q^2+q^3)x^2}.\label{typedgen}
\end{align}
Thus, the sequences $\{\P_{B_r}(q)\}$, $\{\P_{C_r}(q)\}$ and $\{\P_{D_r}(q)\}$ satisfy the recurrence relations
\begin{align*}
\P_{B_r}(q) &= (2+2q+q^2)\P_{B_{r-1}}(q) - (1+2q+q^2+q^3)\P_{B_{r-2}}(q), \mbox{ for $r \geq 4$, }\\
\P_{C_r}(q) &= (2+2q+q^2)\P_{C_{r-1}}(q) - (1+2q+q^2+q^3)\P_{C_{r-2}}(q), \mbox{ for $r \geq 3$,}\\
\P_{D_r}(q) &= (2+2q+q^2)\P_{D_{r-1}}(q) - (1+2q+q^2+q^3)\P_{D_{r-2}}(q),\mbox{ for $r \geq 6$}.
\end{align*}
Consequently, there are functions $\alpha_1(q),\;\alpha_2(q),\;\gamma_1(q),\;\gamma_2(q),\;\delta_1(q),\;\delta_2(q),\;\beta_1(q),\;\beta_2(q)$ such that 
\begin{align*}
\P_{B_r}(q) &= \alpha_1(q) \cdot 
\left(\beta_1(q)\right)^{r-2} + \alpha_2(q) \cdot \left(\beta_2(q)\right)^{r-2} \mbox{ for every $r \geq 2$},\\
\P_{C_r}(q) &= \gamma_1(q) \cdot 
\left(\beta_1(q)\right)^{r-1} + \gamma_2(q) \cdot \left(\beta_2(q)\right)^{r-1}\mbox{ for every $r \geq 1$},\\
\P_{D_r}(q) &= \delta_1(q) \cdot 
\left(\beta_1(q)\right)^{r-4} + \delta_2(q) \cdot \left(\beta_2(q)\right)^{r-4} \mbox{ for every $r \geq 4$}.
\end{align*}  To find the explicit 
formulas for $\P_{B_r}(q)$, $\P_{C_r}(q)$ and $\P_{D_r}(q)$ in terms of $r$, we determine the functions $\alpha_1(q),\;\alpha_2(q),\;\gamma_1(q),\;\gamma_2(q),\;\delta_1(q),\;\delta_2(q),\;\beta_1(q),\;\beta_2(q)$.

From the generating functions in Equations \eqref{typebgen}, \eqref{typecgen} and \eqref{typedgen}, $\beta_1(q)$ and $\beta_2(q)$ are roots, in the variable $\lambda$ in terms of $q$, of the polynomial $f_{\lambda}(q)=\lambda^2-(2+2q+q^2)\lambda+(1+2q+q^2+q^3)$.  Hence, without loss of generality,
\[
\beta_1(q) = \dfrac{2+2q+q^2+q\sqrt{q^2+4}}{2}, \hspace{0.1in} \beta_2(q) = \dfrac{2+2q+q^2-q\sqrt{q^2+4}}{2}
.\]
We now continue on a case by case basis.\\
\noindent
{\bf Type B:} Since $\P_{B_2}(q) = q^3+q^2+q$ and $\P_{B_3}(q)=q^5+2q^4+4q^3+3q^2+q$, we can determine $\alpha_1(q),\alpha_2(q)$ by solving the 
system
\begin{align*}
q^3+q^2+q&= \alpha_1(q)  + \alpha_2(q)  \\
q^5+2q^4+4q^3+3q^2+q &= \alpha_1(q) \cdot\beta_1(q)+\alpha_2(q)\cdot \beta_2(q).
\end{align*}
This yields
\begin{align*}
\alpha_1(q) &= q \cdot \dfrac{q^4 + q^2\left(\sqrt{q^2+4}+5\right) + 
q\left(3\sqrt{q^2+4}+4\right)+2\left(\sqrt{q^2+4}+2\right)+q^3\left(\sqrt{q^2+4}+1\right)}{2(q^2+4)}, \\
\alpha_2(q) &= q \cdot \dfrac{q^4-q^2\left(\sqrt{q^2-4}-5\right)+q\left(4-3\sqrt{q^2+4}\right)-2\left(\sqrt{q^2+4}-2\right) - 
q^3\left(\sqrt{q^2+4}-1\right)}{2(q^2+4)}.
\end{align*}
\noindent
{\bf Type C:} Since $\P_{C_1}(q) = q$ and $\P_{C_2}(q) = q^3+q^2+q$, we can determine $\gamma_1(q),\gamma_2(q)$ by solving the system
\begin{align*}
q &= \gamma_1(q)  + \gamma_2(q)  \\
q^3+q^2+q &= \gamma_1(q)\cdot\beta_1(q)+\gamma_2(q)\cdot\beta_2(q).
\end{align*}
This yields
\[
\gamma_1(q) = \dfrac{q\left(q^2 + q\sqrt{q^2+4}+4\right)}{2(q^2+4)},\hspace{3mm}
\gamma_2(q) = \dfrac{q\left(q^2 - q\sqrt{q^2+4}+4\right)}{2(q^2+4)}.
\]
\noindent
{\bf Type D:} Since $\P_{D_4}(q) = q^5+3q^4+6q^3+4q^2+q$ and $\P_{D_5}(q) = q^7+4q^6+11q^5+17q^4+15q^3+6q^2+q$, we can determine $\delta_1(q),\delta_2(q)$ by solving the system
\begin{align*}
 q^5+3q^4+6q^3+4q^2+q &= \delta_1(q)  + \delta_2(q)  \\
q^7+4q^6+11q^5+17q^4+15q^3+6q^2+q &= \delta_1(q)\cdot\beta_1(q)+\delta_2(q)\cdot\beta_2(q).
\end{align*}
This yields
\begin{align*}
\delta_1(q) &=q\cdot\frac{ 2 + 9 q + 12 q^2 + 8 q^3 + 3 q^4 + q^5 +  (1 + 4 q + 6 q^2 + 3 q^3 + q^4)\sqrt{q^2+4}}{2 \sqrt{
 q^2+4}},\\
\delta_2(q) &=q\cdot\frac{-2 - 9 q - 12 q^2 - 8 q^3 - 3 q^4 - q^5 +(1 + 4 q + 6 q^2 + 3 q^3 + q^4)\sqrt{q^2+4}}{2\sqrt{q^2+4}}.\qedhere
\end{align*}
\end{proof}
\bibliography{HIO.bib}{}

\begin{thebibliography}{10}

\bibitem{BBCV}
W.~Baldoni, M.~Beck, C.~Cochet, and M.~Vergne.
\newblock Volume computation for polytopes and partition functions for
  classical root systems.
\newblock {\em Discrete Comput. Geom.}, 35(4):551--595, 2006.

\bibitem{Baldoni}
W.~Baldoni and M.~Vergne.
\newblock Kostant partitions functions and flow polytopes.
\newblock {\em Transform. Groups}, 13(3-4):447--469, 2008.

\bibitem{BeckRobins}
M.~Beck and S.~Robins.
\newblock {\em Computing the continuous discretely}.
\newblock Undergraduate Texts in Mathematics. Springer, New York, 2007.
\newblock Integer-point enumeration in polyhedra.

\bibitem{BZ2}
A.~D. Berenstein and A.~V. Zelevinsky.
\newblock Tensor product multiplicities and convex polytopes in partition
  space.
\newblock {\em J. Geom. Phys.}, 5(3):453--472, 1988.

\bibitem{BGR}
S.~Billey, V.~Guillemin, and E.~Rassart.
\newblock A vector partition function for the multiplicities of
  {$\germ{sl}_k\Bbb C$}.
\newblock {\em J. Algebra}, 278(1):251--293, 2004.

\bibitem{juggle}
S.~Butler and R.~Graham.
\newblock Enumerating (multiplex) juggling sequences.
\newblock {\em Ann. Comb.}, 13(4):413--424, 2010.

\bibitem{cochet}
C.~Cochet.
\newblock Vector partition function and representation theory.
\newblock {\em Conference Proceedings Formal Power Series and Algebraic
  Combinatorics}, page~12, 2005.

\bibitem{deckhart}
Robert~W. Deckhart.
\newblock On the combinatorics of {K}ostant's partition function.
\newblock {\em J. Algebra}, 96(1):9--17, 1985.

\bibitem{Readdy}
R.~Ehrenborg and M.~Readdy.
\newblock Juggling and applications to {$q$}-analogues.
\newblock In {\em Proceedings of the 6th {C}onference on {F}ormal {P}ower
  {S}eries and {A}lgebraic {C}ombinatorics ({N}ew {B}runswick, {NJ}, 1994)},
  volume 157, 1-3, pages 107--125, 1996.

\bibitem{FP}
J.~Fern\'{a}ndez N\'{u}\~{n}ez, W.~Garc\'{i}a~Fuertes, and A.~M. Perelomov.
\newblock Generating functions and multiplicity formulas: the case of rank two
  simple lie algebras.
\newblock {\em arXiv}, 2015.

\bibitem{GW}
R.~Goodman and N.~R. Wallach.
\newblock {\em Symmetry, representations, and invariants}, volume 255 of {\em
  Graduate Texts in Mathematics}.
\newblock Springer, Dordrecht, 2009.

\bibitem{Gupta}
R.~K. Gupta.
\newblock Characters and the q-analog of weight multiplicity.
\newblock {\em Journal of the London Mathematical Society}, s2-36(1):68--76,
  1987.

\bibitem{PH2}
P.~E. Harris.
\newblock Kostant's weight multiplicity formula and the fibonacci numbers.
\newblock {\em arXiv}, 2011.

\bibitem{PH}
P.~E. Harris.
\newblock On the adjoint representation of {$\germ s \germ l_n$} and the
  {F}ibonacci numbers.
\newblock {\em C. R. Math. Acad. Sci. Paris}, 349(17-18):935--937, 2011.

\bibitem{PH3}
P.~E. Harris.
\newblock {\em Combinatorial problems related to {K}ostant's weight
  multiplicity formula}.
\newblock ProQuest LLC, Ann Arbor, MI, 2012.
\newblock Thesis (Ph.D.)--The University of Wisconsin - Milwaukee.

\bibitem{HIW}
P.~E. Harris, E.~Insko, and L.~K. Williams.
\newblock The adjoint representation of a classical lie algebra and the support
  of kostant's weight multiplicity formula.
\newblock {\em Journal of Combinatorics}, 2015.

\bibitem{KMF}
B.~Kostant.
\newblock A formula for the multiplicity of a weight.
\newblock {\em Trans. Amer. Math. Soc.}, 93:53--73, 1959.

\bibitem{Kostant}
B.~Kostant.
\newblock The principal three-dimensional subgroup and the {B}etti numbers of a
  complex simple {L}ie group.
\newblock {\em Amer. J. Math.}, 81:973--1032, 1959.

\bibitem{LL}
G.~Lusztig.
\newblock Singularities, character formulas, and a {$q$}-analog of weight
  multiplicities.
\newblock In {\em Analysis and topology on singular spaces, {II}, {III}
  ({L}uminy, 1981)}, volume 101 of {\em Ast\'erisque}, pages 208--229. Soc.
  Math. France, Paris, 1983.

\bibitem{MM}
K.~M{\'e}sz{\'a}ros and A.~H. Morales.
\newblock Flow polytopes of signed graphs and the {K}ostant partition function.
\newblock {\em Int. Math. Res. Not. IMRN}, (3):830--871, 2015.

\bibitem{BS}
J.~R. Schmidt and A.~M. Bincer.
\newblock The {K}ostant partition function for simple {L}ie algebras.
\newblock {\em J. Math. Phys.}, 25(8):2367--2373, 1984.

\bibitem{Tate}
T.~Tate and S.~Zelditch.
\newblock Lattice path combinatorics and asymptotics of multiplicities of
  weights in tensor powers.
\newblock {\em J. Funct. Anal.}, 217(2):402--447, 2004.

\end{thebibliography}
\bibliographystyle{plain}

\end{document}